\documentclass[12pt]{amsart}
\usepackage{amssymb}
\usepackage[all]{xy}
\usepackage{color}
\usepackage{amsmath}
\input xy
\xyoption{all}
\usepackage{pgf, tikz}
\usepackage{mathtools}
\usepackage{amsfonts}
\usepackage{amsthm}
\usepackage{amscd}
\newtheorem{lemma}{Lemma}[section]

\newtheorem{theorem}[lemma]{Theorem}

\theoremstyle{definition}

\newtheorem{definition}[lemma]{Definition}

\newtheorem{example}[lemma]{Example}

\newtheorem{conjecture}[lemma]{Conjecture}
\newcommand{\bsm}{\begin{smallmatrix}}
\newcommand{\esm}{\end{smallmatrix}}
\newcommand{\bbm}{\begin{matrix}}
\newcommand{\ebm}{\end{matrix}}
\DeclareMathOperator{\Hom}{Hom}
\DeclareMathOperator{\Ext}{Ext}
\DeclareMathOperator{\End}{End}

\begin{document}

\title[Maximal forward hom-orthogonal sequences]{Maximal forward hom-orthogonal sequences for cluster-tilted algebras of finite type}
\newcommand\shortTitle{Maximal forward hom-orthogonal sequences}
\author{Alireza Nasr-Isfahani}
\address{Department of Mathematics,
University of Isfahan,
P.O. Box: 81746-73441, Isfahan, Iran and School of Mathematics, Institute for Research in Fundamental Sciences (IPM), P.O. Box: 19395-5746, Tehran,
Iran}
\email{nasr$_{-}$a@sci.ui.ac.ir / nasr@ipm.ir}



\begin{abstract} Let $\Lambda$ be a cluster-tilted algebra of finite type over an algebraically closed field and $B$ be one of the associated tilted algebras. We show that the $B$-modules, ordered form right to left in the Auslander-Reiten quiver of $\Lambda$ form a maximal forward hom-orthogonal sequence of $\Lambda$-modules whose dimension vectors form the $c$-vectors of a maximal green sequence for $\Lambda$. Thus we give a proof of Igusa-Todorov's conjecture.
\end{abstract}

\maketitle


\section{Introduction}

A maximal green sequence is a certain sequence of quiver mutations were introduced by Keller \cite{K} in order to obtain quantum dilogarithm identities and refined Donaldson-Thomas invariants of Kontsevich and Soibelman. At the same time, the same sequences of quiver mutations are studied in theoretical physics (see \cite{ACCERV, CCV}). Maximal green sequences are also connected to representation theory \cite{BDP} and cluster algebras \cite{FZ1}.

The Cluster-tilted algebras have been introduced and investigated by Buan, Marsh and Reiten
\cite{BMR1}. They are the endomorphism rings of cluster-tilting objects in cluster categories. Cluster-tilted algebras have a key role in the study of cluster categories.
Also an important connection between cluster algebras and
cluster-tilted algebras was established in \cite{BMR2} and
\cite{CCS2}.

Igusa in \cite{I} studied maximal green sequences for cluster-tilted algebras of finite type. It is known that a cluster-tilted algebra of finite type over an algebraically closed field is isomorphic to the Jacobian algebra $J(Q, W)$ of a quiver with potential of finite type \cite{DWZ}. Let $Q$ be a finite type quiver with nondegenerate potential $W$ over any field and $\Lambda=J(Q, W)$ be the Jacobian algebra. Let $\beta_1, \cdots, \beta_m$ be a finite sequence of elements of $\mathbb{N}^n$ where $n$ is the number of vertices of $Q$. Igusa in \cite{I} proved that the following conditions are equivalent:
\begin{itemize}
\item[(1)] $\beta_1, \cdots, \beta_m$ are the $c$-vectors of a maximal green sequence for $\Lambda$.
\item[(2)] There exist Schurian $\Lambda$-modules $M_1, \cdots, M_m$ with dimension vectors $\underline{dim} M_i=\beta_i$ such that:
\begin{itemize}
\item[(a)] $Hom_\Lambda(M_i, M_j)=0$ for all $i<j$.
\item[(b)] No other module can be inserted into this sequence preserving $(a)$.
\end{itemize}
\item[(3)] There exists a generic green path $\gamma$ crossing a finite number of stability walls $D(M_1), \cdots D(M_m)$ so that $\underline{dim} M_i=\beta_i$
\end{itemize}

The sequence $M_1, \cdots, M_m$ of Schurian $\Lambda$-modules in $(2)$ is called a maximal forward hom-orthogonal sequence \cite{I}. Therefore, for any Jacobian algebra $\Lambda=J(Q, W)$ of finite type, there is a one to one correspondence between the maximal green sequences and the maximal forward hom-orthogonal sequences of Schurian modules. Igusa and Todorov in the appendix of \cite{I} studied the maximal forward hom-orthogonal sequences for cluster-tilted algebras of type $A_n$. They used a theorem due to Assem et al. \cite{ABS} and Zhu \cite{Z}, which gives an equivalent definition of cluster-tilted algebras. According to that theorem, any cluster-tilted algebra is isomorphic to the relation-extension of some tilted algebra. Let $B$ be a tilted algebra of type $A_n$ and $\Lambda$ be the relation-extension of $B$. Igusa and Todorov proved that the indecomposable $B$-modules, ordered from right to left in the Auslander-Reiten quiver of $\Lambda$, form a maximal forward hom-orthogonal sequence of $\Lambda$-modules. They also posted the following conjecture:\\

\begin{conjecture}\label{con}(Igusa-Todorov)
Let $\Lambda$ be a cluster-tilted algebra of finite type and let $B$ be one of the associated tilted algebras.
\begin{itemize}
\item[(a)] The $B$-modules, ordered from right to left in the Auslander-Reiten quiver of $\Lambda$ form a maximal forward hom-orthogonal sequence of $\Lambda$-modules whose dimension vectors form the $c$-vectors of a maximal green sequence for $\Lambda$.
\item[(b)] The longest maximal green sequence for $\Lambda$ is given in this way.
\end{itemize}
\end{conjecture}

In this paper by using cluster repetitive algebras we prove the part $(a)$ of the Igusa-Todorov's conjecture.
Before proving our main result in Section~\ref{sec:results} we provide the necessary background in the
following section.


\section{Preliminaries}

A quiver $Q$ is a quadruple $(Q_0, Q_1, s, t)$, where $Q_0$ is a set of vertices, $Q_1$ is a set of arrows and two functions $s, t:Q_1\rightarrow Q_0$ assign to each arrow $\alpha\in Q_1$ its source $s(\alpha)$ and its target $t(\alpha)$, respectively. In this note, quivers will be finite and connected. A quiver is called cluster quiver if it has no loops or oriented $2$-cycles.

\subsection{Ice quivers and mutations.} An ice quiver is a pair $(Q, F)$ where $Q$ is a quiver and $F\subset Q_0$ is a subset of vertices called frozen vertices such that there are no arrows between them. Elements of $Q_0\backslash F$ are called mutable vertices.
Let $Q$ be a cluster quiver, $(Q, F)$ be an ice quiver and $l\in Q_0$ be a mutable vertex. The mutation of $(Q, F)$ at $l$ is a new ice quiver $(\mu_l(Q), F)$, where $\mu_l(Q)$ is obtained from $Q$ by applying the following three steps:\\
\begin{itemize}
\item[(1)] For every path $i\rightarrow l\rightarrow j$ in $Q$, add a new arrow $i\rightarrow j$.
\item[(2)] Reverse the direction of all arrows incident to $l$ in $Q$.
\item[(3)] Remove any $2$-cycles created and remove all the arrows created between frozen vertices.
\end{itemize}
Two ice quivers are called mutation equivalent if one obtainable from the other by finitely many mutations at mutable vertices. Since mutation is involution, mutation equivalent defines an equivalence relation on the set of ice quivers. The equivalence class of an ice quiver $(Q, F)$, up to isomorphism of quivers that fixes the frozen vertices, is called mutation class of $(Q, F)$ and is denoted by  $Mut(Q, F)$.
\subsection{Maximal green sequences} Let $Q$ be a cluster quiver with the vertex set $Q_0=\{1, 2, \cdots, n\}$. The framed quiver associated with $Q$ is the ice quiver $\widehat{Q}$ where $\widehat{Q}_0:=Q_0\sqcup \{n+1, n+2, \cdots, 2n\}$, $F=\{n+1, n+2, \cdots, 2n\}$ and $\widehat{Q}_1:=Q_1\sqcup \{i\rightarrow i+n| i\in Q_0\}$. For any $1\leq i\leq n$, we write $i'=n+i$. Let $R\in Mut(\widehat{Q})$. A mutable vertex $i$ of $R$ is called green (resp., red) if there are no arrows in $R$ of the form $j'\rightarrow i$ (resp., $i\rightarrow j'$) for some $1\leq j\leq n$. The celebrated theorem of Derksen et al. \cite[Theorem 1.7]{DWZ1}, implies that any mutable vertex of $R$ is either green or red.

\begin{definition} (\cite{K}) A green sequence for $Q$ is a sequence $\mathbf{i}=(i_1, \cdots, i_l)$ of mutable vertices of $\widehat{Q}$ such that $i_1$ is green in $\widehat{Q}$ and $i_r$ is green in $\mu_{i_{r-1}}\circ\cdots \circ\mu_{i_1}(\widehat{Q})$, for each $2\leq r\leq l$. $l$ is called the length of the sequence $\mathbf{i}$. A green sequence $\mathbf{i}=(i_1, \cdots, i_l)$ is called maximal if every non-frozen vertex in $\mu_{\mathbf{i}}(\widehat{Q})=\mu_{i_{l}}\circ\cdots \circ\mu_{i_1}(\widehat{Q})$ is red.
\end{definition}

Let $Q$ be a cluster quiver with the vertex set $Q_0=\{1, 2, \cdots, n\}$. The exchange matrix $E_{\widehat{Q}}$ of $\widehat{Q}$ is $n\times 2n$ integer matrix with $(i, j)$-th entry $e_{ij}$ equal to the number of arrows from $i$ to $j$ minus the number of arrows from $j$ to $i$ in the quiver $\widehat{Q}$. For any $R\in Mut(\widehat{Q})$ the $n\times n$ submatrix of $E_R$ containing its last $n$ columns is called $c$-matrix of $Q$. A row vector of a $c$-matrix is called a $c$-vector.

\subsection{Cluster-tilted algebras} Let $H$ be a hereditary finite dimensional $k$-algebra, where $k$ is
an algebraically closed field and let $\mathcal{D} = D^b({\rm mod} H)$ be
the bounded derived category of finitely generated right $H$-modules
with shift functor $[1]$. Also, let $\tau$ be the $AR$-translation
in $\mathcal{D}$. The cluster category is defined as the orbit
category $\mathcal{C}_H = \mathcal{D} /F$, where $F = \tau^{-1}[1]$.
The objects of $\mathcal{C}_H$ are the same as the objects of
$\mathcal{D}$, but maps are given by $\Hom_{\mathcal{C}_H}(X, Y )
=\bigoplus_{i\in \mathbb{Z}} \Hom_\mathcal{D}(X, F^iY)$ (see \cite{BMRRT}). An object $\widetilde{T}$ in
$\mathcal{C}_H$ is called cluster-tilting provided for any object
$X$ of $\mathcal{C}_H$, we have $\Ext^1_{\mathcal{C}_H}(\widetilde{T}, X) = 0$ if
and only if $X$ lies in the additive subcategory ${\rm add}(\widetilde{T})$ of
$\mathcal{C}_H$ generated by $\widetilde{T}$. Let $\widetilde{T}$ be a cluster-tilting
object in $\mathcal{C}_H$. The cluster-tilted algebra associated to
$\widetilde{T}$ is the algebra $\End_{\mathcal{C}_H}(\widetilde{T})^{op}$ \cite{BMR1}. It is known that a cluster-tilted algebra $\End_{\mathcal{C}_H}(\widetilde{T})^{op}$ is of finite representation type if and only if $H$ is Morita equivalent to the path algebra of a simply-laced Dynkin quiver \cite{BMR1}.

Let $\Lambda$ be a cluster-tilted algebra over
an algebraically closed field $k$. Then (up to Morita equivalence) $\Lambda$ is of the form $kQ/I$,
where $Q$ is a finite quiver and $I$ is some admissible ideal in the path algebra $kQ$. It is known that $Q$ is a cluster quiver \cite{ABS}. A (maximal) green sequence for $\Lambda$ is a (maximal) green sequence for $Q$.

Assem, Br\"{u}stle and Schiffler \cite{ABS} and Zhu \cite{Z} independently provided a
characterization of
cluster-tilted algebras. They proved that an algebra $\Lambda$ is
cluster-tilted if and only if there exists a tilted algebra $B$ such
that $\Lambda\cong B^c$, where $B^c$ is trivial extension algebra
$B^c = B \ltimes \Ext^2_B(DB,B)$, with $D = \Hom(-, k)$ the
$k$-duality. Recall that a $k$-algebra $B$ is said to be tilted
provided $B$ is the endomorphism ring of a tilting $H$-module $T$,
where $H$ is a finite dimensional hereditary $k$-algebra.

\subsection{Maximal forward hom-orthogonal sequence} Let $\Lambda$ be a $k$-algebra. For any finitely generated right $\Lambda$-module $M$, let $\mathcal{F}(M)=M^\bot=\{X\in mod\Lambda | Hom_\Lambda(M, X)=0\}$ and $\mathcal{G}(M)={^\bot}\mathcal{F}(M)=\{X\in mod\Lambda | Hom_\Lambda(X, Y)=0 \text{ for all } Y\in \mathcal{F}(M)\}$. Therefore $\mathcal{F}(M)=\mathcal{G}(M)^\bot$ and $(\mathcal{G}(M), \mathcal{F}(M))$ is a torsion pair in $mod\Lambda$.

A $\Lambda$-module $M$ is called Schurian if its endomorphism ring is a division ring.

\begin{definition} (Definition 2.2 of \cite{I}) A forward hom-orthogonal sequence in $mod\Lambda$ is a finite sequence of Schurian modules $M_1, M_2, \cdots, M_m$ that
\begin{itemize}
\item[(1)] $Hom_\Lambda(M_i, M_j)=0$ for all $1\leq i<j\leq m$.
\item[(2)] The sequence is maximal in $\mathcal{G}(M)$ where $M=M_1\oplus \cdots\oplus M_m$.
\end{itemize}
A forward hom-orthogonal sequence is maximal if $\mathcal{G}(M)=mod\Lambda$.
\end{definition}

\subsection{Cluster repetitive algebra} The cluster repetitive algebra is a certain Galois covering of a cluster-tilted algebra defined by Assem, Br\"{u}stle and Schiffler in \cite{ABS1}. Let $B$ be a tilted algebra, $B_i=B$ and $E_i=Ext^2_B(DB, B)$ with $D = \Hom(-, k)$ the
$k$-duality, for each $i\in \mathbb{Z}$. The following locally finite dimensional algebra without identity
$$\check{B}=\left(\begin{array}{ccccc}
\ddots & &&0 & \\
\ddots &B_{-1}& & & \\
 &E_0&B_0&& \\
 & &E_1&B_1&\\
 & 0&&\ddots &\ddots \\
\end{array}\right)$$
where matrices have only finitely many non-zero entries, and the multiplication is induced from that of $B$, the $B$-$B$-bimodule $Ext^2_B(DB, B)$ and the zero map $Ext^2_B(DB, B)\otimes_BExt^2_B(DB, B)\rightarrow 0$, is called the cluster repetitive algebra \cite{ABS1}.

Let $\Lambda=B \ltimes \Ext^2_B(DB,B)$ be a cluster-tilted algebra. The identity maps $B_i\rightarrow B_{i-1}$ and $E_i\rightarrow E_{i-1}$ induce an isomorphism $\varphi$ of $\check{B}$. The orbit category $\check{B}/(\varphi)$ which inherits the structure of $k$-algebra from $\check{B}$ is isomorphic to $\Lambda$ \cite{ABS1}. The projection functor $\check{B}\rightarrow \Lambda$ is a Galois covering functor with infinite cyclic group generated by $\varphi$ and we denote by $\pi: mod\check{B}\rightarrow mod\Lambda$ the corresponding push-down functor.


\section{Main Result}\label{sec:results}

Let $\Lambda$ be an arbitrary cluster-tilted algebra of finite type. It is known that there exists a quiver $Q$ of Dynkin type such that $H=kQ$, and a tilting $H$-module $T$ such that $\Lambda\cong B\ltimes E$ where $B=\End_H(T)$ and $E=\Ext^2_B(DB,B)$ (see \cite{ABS} or \cite{Z}). Let $(\mathcal{T}(T_H), \mathcal{F}(T_H))$ be the torsion pair induced by $T$ in $modH$, where $\mathcal{T}(T_H)=\{X\in modH| Ext^1_H(T, X)=0\}$ and $\mathcal{F}(T_H)=\{X\in modH| Hom_H(T, X)=0\}$. The tilting $H$-module $T$ induces a splitting torsion pair $(\mathcal{X}(T_H), \mathcal{Y}(T_H))$ in $modB$, where $\mathcal{X}(T_H)=\{X\in modB| X\otimes_BT=0\}$ and $\mathcal{Y}(T_H)=\{Y\in modB| Tor^B_1(Y, T)=0\}$. Let $\check{B}$ be the corresponding cluster repetitive algebra. The $\check{B}$-modules are of the form $(M_i, \alpha_i)_{i\in \mathbb{Z}}$, with $B$-modules $M_i$ and $B$-homomorphisms $\alpha_i:M_i\otimes_BE\rightarrow M_{i-1}$.

\begin{lemma}\label{Lem} Let $M=(M_i, \alpha_i)_{i\in \mathbb{Z}}$ be an indecomposable $\check{B}$-module of finite length. Then there exists $t\in \mathbb{Z}$ such that $M_{t-1}\in\mathcal{X}(T_H)$, $M_{t}\in\mathcal{Y}(T_H)$ and $M_i=0$ for each $i\neq t, t-1$.
\end{lemma}
\begin{proof} There exist $X_i\in \mathcal{X}(T_H)$ and $Y_i\in \mathcal{Y}(T_H)$ such that $M_i=X_i\oplus Y_i$, for each $i\in \mathbb{Z}$. By the Brenner-Butler theorem, $X_i\cong Ext^1_H(T, X'_i)$ for some $X'_i\in \mathcal{F}(T_H)$. Also we have $E= Ext^1_H(T, \tau^{-1}(T))$ (see the proof of \cite[Theorem 3.4]{ABS}). Then $Hom_B(E, Ext^1_H(T, X'_i))=Hom_B(Ext^1_H(T, \tau^{-1}(T)), Ext^1_H(T, X'_i))\cong Hom_H(\tau^{-1}(T), X'_i)$ and by \cite[Corollary IV.2.15]{ASS}, $Hom_H(\tau^{-1}(T), X'_i)\cong Hom_H(T, \tau(X_i))$. Therefore $Hom_B(E, Ext^1_H(T, X'_i))\cong\\ Hom_H(T, \tau(X_i))=Hom_H(T, t(\tau(X_i)))\in\mathcal{Y}(T_H)$, where $t(\tau(X_i))$ is a torsion submodule of $\tau(X_i)$. By the adjoint formula, $Hom_B(X_i\otimes_BE, X_{i-1})\cong Hom_B(X_i, Hom_B(E, X_{i-1}))$. But $Hom_B(E, X_{i-1})\in\mathcal{Y}(T_H)$ and $X_i\in \mathcal{X}(T_H)$, and so $Hom_B(X_i\otimes_BE, X_{i-1})=0$. Also $Hom_B(X_i\otimes_BE, Y_{i-1})\cong Hom_B(X_i, Hom_B(E, Y_{i-1}))$ and $E= Ext^1_H(T, \tau^{-1}(T))\in \mathcal{X}(T_H)$, then $Hom_B(X_i\otimes_BE, Y_{i-1})=0$. Therefore the restriction of $\alpha_i$ to $X_i\otimes_BE$ is zero. Since $M_i$ is a finitely generated $B$-module, there exists an epimorphism $f:B^r\rightarrow M_i$. Then we have an epimorphism $f:B^r\otimes_BE\rightarrow M_i\otimes_BE$ and so $M_i\otimes_BE\in \mathcal{X}(T_H)$. Therefore the image of $\alpha_i$ contains in $X_{i-1}$. Then the indecomposability of $M$ implies that $M_{t-1}=X_{t-1}$ and $M_t=Y_t$, for some $t\in \mathbb{Z}$ and $M_i=0$ for each $i\neq t, t-1$ and the result follows.
\end{proof}

Now, we are ready to prove our main theorem.

\begin{theorem}\label{Thm} Let $\Lambda$ be a cluster-tilted algebra of finite type and $B$ be one of the associated tilted algebras. Then the indecomposable $B$-modules, ordered from right to left in the Auslander-Reiten quiver of $\Lambda$, form a maximal forward hom-orthogonal sequence of $\Lambda$-modules whose dimension vectors form the $c$-vectors of a maximal green sequence for $\Lambda$.
\end{theorem}

\begin{proof} Let $H=kQ$ be a representation finite hereditary algebra, $T$ be a tilting $H$-module, $B=\End_H(T)$ and $E=Ext^2_B(DB,B)$ such that $\Lambda\cong B\ltimes E$. Let $M_i$ be the indecomposable $B$-modules numbered from right to left in the Auslander-Reiten quiver of $\Lambda$. We have $Hom_\Lambda(M_i, M_j)=0$ for $i<j$ (see the proof of \cite[Lemma 3.24]{I}). Let $M$ be an indecomposable $\Lambda$-module which is not a $B$-module. By using \cite[Theorem 1.1]{I}, it is enough to show that the $\Lambda$-module $M$ cannot be in any forward hom-orthogonal sequence which contains all indecomposable $B$-modules. By Lemma \ref{Lem}, $\check{B}$ is locally support finite and so by the density theorem of Dowbor and Skowro\'{n}ski \cite{DS}, there exists indecomposable $\check{B}$-module $\widetilde{M}$ such that $\pi(\widetilde{M})\cong M$. By Lemma \ref{Lem}, there exists $t\in\mathbb{Z}$ such that $\widetilde{M}=(\widetilde{M}_i, \alpha_i)_{i\in \mathbb{Z}}$, where $\widetilde{M}_{t-1}\in\mathcal{X}(T_H)$, $\widetilde{M}_{t}\in\mathcal{Y}(T_H)$, $\widetilde{M}_i=0$ for each $i\neq t, t-1$ and $\alpha_t:\widetilde{M}_t\otimes_BE\rightarrow \widetilde{M}_{t-1}$. Let $\widetilde{L}=(\widetilde{L}_i, \beta_i)_{i\in \mathbb{Z}}$, where $\widetilde{L}_{t-1}=\widetilde{M}_{t-1}$ and $\widetilde{L}_i=0$ for each $i\neq t-1$ and $\widetilde{N}=(\widetilde{N}_i, \gamma_i)_{i\in \mathbb{Z}}$, where $\widetilde{N}_{t}=\widetilde{M}_{t}$ and $\widetilde{N}_i=0$ for each $i\neq t$. Then $\widetilde{L}$ and $\widetilde{N}$ are $\check{B}$-modules and we have an exact sequence $0\longrightarrow \widetilde{L}\buildrel {\widetilde{f}} \over \longrightarrow \widetilde{M}\buildrel {\widetilde{g}} \over\longrightarrow\widetilde{N}\longrightarrow 0$ in $mod\check{B}$ which induces an exact sequence
$0\longrightarrow \widetilde{M}_{t-1}\buildrel {f} \over \longrightarrow M\buildrel {g} \over\longrightarrow\widetilde{M}_{t}\longrightarrow 0$ in $mod\Lambda$. Since $\widetilde{M}_{t-1}$ and $\widetilde{M}_{t}$ are $B$-modules and $M$ is not a $B$-module, $f$ and $g$ are nonzero homomorphisms. Finally, since $\widetilde{M}_{t-1}\in\mathcal{X}(T_H)$, $\widetilde{M}_{t}\in\mathcal{Y}(T_H)$ and $f$ and $g$ are nonzero, the $\Lambda$-module $M$ cannot be in any forward hom-orthogonal sequence which contains all indecomposable $B$-modules and the result follows.
\end{proof}

The following example shows that there exist a cluster-tilted algebra $\Lambda$ of finite type and an associated tilted algebra $B$ such that the part $(b)$ of the conjecture \ref{con} is not true for $\Lambda$ and $B$.

\begin{example}\label{exa}(\cite[Examples 2.3 and 3.3]{BN}) Let $\Lambda=kQ/I$ be a cluster-tilted algebra, where the quiver $Q$ given by
$$\hskip.5cm 
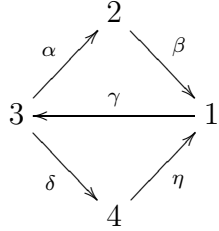
\xymatrix{
&{2} \ar@{<-}[dl]_{\alpha}\ar[dr]^{\beta}&\\
{3}\ar[dr]_{\delta}&&{1}\ar[ll]_{\gamma}\\
&{4}\ar[ur]_{\eta}&
}
\hskip.5cm$$
and $I=<\alpha\beta-\delta\eta, \eta\gamma, \gamma\delta, \gamma\alpha, \beta\gamma>$ is an admissible ideal of $kQ$. Let $H=k\Gamma$ be a representation finite hereditary algebra, where the quiver $\Gamma$ given by
$$\hskip.5cm 
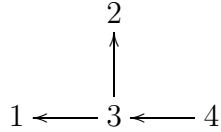
\xymatrix{
&{2}&\\
{1}&{3}\ar[l]\ar[u]&{4}\ar[l]\\
}
\hskip.5cm
$$
and $T=P_1\oplus P_2\oplus S_4\oplus P_4$ be a tilting $H$-module, where $P_i$ is the indecomposable projective $H$-module corresponding to the vertex $i$ of $\Gamma$ and $S_i$ is the simple $H$-module corresponding to the vertex $i$ of $\Gamma$. Then $\Lambda\cong B\ltimes \Ext^2_B(DB, B)$, where $B=End_H(T)$. The Auslander-Reiten quiver of $\Lambda$ given by
$$
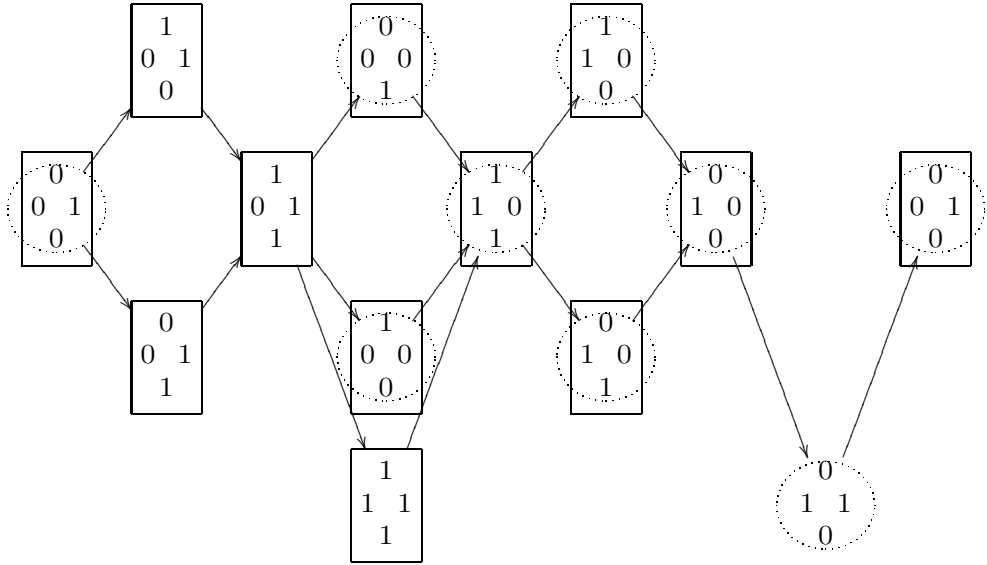
\resizebox{12.75cm}{.75cm}{
\xymatrix@-3mm{
&*+[F]{\bbm1\\0\hskip.3cm 1\\0\ebm}\ar[dr]&&*+[F][o][F.]{\bbm0\\0\hskip.3cm 0\\1\ebm}\ar[dr]&&*+[F][o][F.]{\bbm1\\1\hskip.3cm 0\\0\ebm}\ar[dr]&&& \\
*+[F][o][F.]{\bbm0\\0\hskip.3cm1\\0\ebm}\ar[dr]\ar[ur]&&*+[F]{\bbm1\\0\hskip.3cm1\\1\ebm}\ar[ur]\ar[dr]\ar[ddr]&&*+[F][o][F.]{\bbm1\\1\hskip.3cm0\\1\ebm}\ar[dr]\ar[ur]&&*+[F][o][F.]{\bbm0\\1\hskip.3cm0\\0\ebm}\ar[ddr]&&*+[F][o][F.]{\bbm0\\0\hskip.3cm1\\0\ebm}\\
&*+[F]{\bbm0\\0\hskip.3cm1\\1\ebm}\ar[ur]&&*+[F][o][F.]{\bbm1\\0\hskip.3cm0\\0\ebm}\ar[ur]&&*+[F][o][F.]{\bbm0\\1\hskip.3cm0\\1\ebm}\ar[ur]&&&\\
&&&*+[F]{\bbm1\\1\hskip.3cm1\\1\ebm}\ar[uur]&&&&*+[o][F.]{\bbm0\\1\hskip.3cm1\\0\ebm}\ar[uur]&\\
}
}
$$
, where indecomposable modules are represented by their dimension vectors and indecomposable $B$-modules in the Auslander-Reiten quiver of $\Lambda$ are represented in circles. Let $m, p$ be the maximum and minimum length of maximal green sequences for $\Lambda$. According to the \cite[Proposition 3.19]{I}, $m+p-4$ is at most equal to 12. Also by \cite[Theorem 6.1]{GMS}, $p=5$ and hence $m\leq 11$. Now let $H'=k\Gamma'$, where the quiver $\Gamma'$ given by
$$\hskip.5cm 
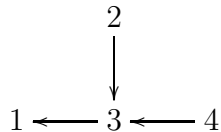
\xymatrix{
&{2}\ar[d]&\\
{1}&{3}\ar[l]&{4}\ar[l]\\
}
\hskip.5cm
$$
, $T'=P_1\oplus P_2\oplus I_1\oplus P_4$ be a tilting $H'$-module, where $I_1$ is the indecomposable injective $H'$-module corresponding to the vertex $1$ of $\Gamma'$ and $B'=End_{H'}(T')$. We have $\Lambda\cong B'\ltimes \Ext^2_{B'}(DB', B')$. There are $11$ indecomposable $B'$-modules in the Auslander-Reiten quiver of $\Lambda$ which represented by squares. Therefore by Theorem \ref{Thm}, the length of the longest maximal green sequence for $\Lambda$ is $11$ but there are $8$ indecomposable $B$-modules in the Auslander-Reiten quiver of $\Lambda$.
\end{example}

Example \ref{exa} suggests the following reformulation of the part $(b)$ of the Conjecture \ref{con}:\\

\begin{conjecture}(Igusa-Todorov)
Let $\Lambda$ be a cluster-tilted algebra of finite type. Then there exists a tilted algebra $B$ such that $\Lambda\cong B\ltimes \Ext^2_B(DB, B)$ and the number of isomorphism classes of indecomposable $\Lambda$-modules which are indecomposable $B$-modules, is equal to the length of the longest maximal green sequence for $\Lambda$.
\end{conjecture}


\section*{acknowledgements}
The author would like to thank the referee for careful reading of the paper and helpful
comments. This research was in part supported by a grant from IPM (No. 96170417).


\begin{thebibliography}{10}

\bibitem{ACCERV}  \textsc{M. Alim, S. Cecotti, C. Cordova, S. Espahbodi, A. Rastogi, C. Vafa}, $\mathcal{N}=2$ quantum field theories and their BPS
quivers, \emph{Adv. Theor. Math. Phys.} \textbf{18} (2014), 27.

\bibitem{ABS}  \textsc{I. Assem, T. Br\"{u}stle, R. Schiffler}, Cluster-tilted algebras as trivial extensions, \emph{Bull. Lond. Math. Soc.} \textbf{40} (2008), 151--162.

\bibitem{ABS1}  \textsc{I. Assem, T. Br\"{u}stle, R. Schiffler}, On the Galois covering of a cluster-tilted algebra, \emph{J. Pure Appl. Algebra} \textbf{213} (2009), 1450--1463.

\bibitem{ASS} \textsc{I. Assem, D. Simson, A. Skowronski}, Elements of the Representation Theory of Associative Algebra,
Techniques of representation theory, vol. 1, London Mathematical Society Student Texts \textbf{65}, Cambridge University Press, Cambridge, 2006.

\bibitem{BN} \textsc{K. Baur, A. Nasr-Isfahani}, Strongness of companion bases for cluster-tilted algebras of finite type, \emph{Proc. Amer. Math. Soc.} \textbf{146}(6) (2018) 2409--2416.

\bibitem{BDP}  \textsc{T. Br\"{u}stle, G. Dupont, M. Pérotin}, On maximal green sequences, \emph{Int. Math. Res. Not. IMRN} (2014), 4547--4586.

\bibitem{BMRRT}  \textsc{A. Buan, R. Marsh, M. Reineke, I. Reiten, G. Todorov}, Tilting theory and cluster combinatorics,
\emph{Adv. Math.} \textbf{204} (2006) 572--618.

\bibitem{BMR1} \textsc{A. Buan, R. Marsh, I. Reiten},  Cluster-tilted algebras, \emph{Trans. Amer.
Math. Soc.}, \textbf{359}(1) (2007), 323-332.

\bibitem{BMR2} \textsc{A. Buan, R. Marsh, I. Reiten},  Cluster mutation via quiver
representations, \emph{Comment. Math. Helv.} \textbf{83}(2) (2008),
143-177.

\bibitem {CCS2} \textsc{P. Caldero, F. Chapoton, R. Schiffler},  Quivers with relations and
cluster-tilted algebras, \emph{Algebr. Represent. Theory}
\textbf{9}(4) (2006), 359-376.

\bibitem {CCV} \textsc{S. Cecotti, C. C´ordova, C. Vafa}, Braids, walls and mirrors, arXiv:1110.2115v1 [hep-th], 2011.

\bibitem{DWZ}  \textsc{H. Derksen, J. Weyman, A. Zelevinsky}, Quivers with potentials and their representations I: Mutations, \emph{Sel. math.} \textbf{14}(1) (2008), 59--119.

\bibitem{DWZ1}  \textsc{H. Derksen, J. Weyman, A. Zelevinsky}, Quivers with potentials and their representations II: applications to cluster algebras, \emph{J. Amer. Math. Soc.} \textbf{23}(3) (2010), 749--790.

\bibitem{DS}  \textsc{P. Dowbor, A. Skowro\'{n}ski}, On Galois coverings of tame algebras, \emph{Arch. Math.} \textbf{44} (1985), 522--529.

\bibitem{FZ1} \textsc{S. Fomin, A. Zelevinsky},  Cluster algebras I: foundations, \emph{J. Amer.
Math. Soc.} \textbf{15}(2) (2002), 497--529.

\bibitem{GMS}\textsc{A. Garver, T. Mcconville, K. Serhiyenko}, Minimal lenght maximal green sequences, \emph{Adv. Appl. Math.} \textbf{96} (2018), 76--138.

\bibitem{I}\textsc{K. Igusa}, Maximal green sequences for cluster-tilted algebras of finite type, arXiv:1706.06503v1 [math.RT], 2017.

\bibitem{K}\textsc{B. Keller}, On cluster theory and quantum dilogarithm identities, in Representations of Algebras and Related Topics, EMS Ser. Congr. Rep. Eur. Math. Soc., pp. 85--116, 2011.

\bibitem{Z}  \textsc{B. Zhu}, Equivalences between cluster categories,
\emph{J. Algebra} \textbf{304} (2006), 832--850.

\end{thebibliography}
\end{document}